\documentclass[twoside]{article} 
\usepackage{graphicx}
\usepackage{amsfonts}
\usepackage[fleqn]{amsmath}
\usepackage{amssymb}
\usepackage{amsthm}
\usepackage{amscd}
\usepackage[T1]{fontenc}
\usepackage{afterpage}  %
\usepackage{fancyhdr}
\usepackage{color}
\theoremstyle{plain}
\newtheorem{theorem}{Theorem}[section]
\newtheorem{lemma}[theorem]{Lemma}
\newtheorem{proposition}[theorem]{Proposition}

\newtheorem{corollary}[theorem]{Corollary}
\theoremstyle{definition}
\newtheorem{definition}[theorem]{Definition}

\theoremstyle{remark}

\begin{document}

\afterpage{\rhead[]{\thepage} \chead[\small W. A. Dudek and R. A. R. Monzo \ \ \ \   ]
{\small The structure of idempotent translatable quasigroups      \ \ \ \ \ \ \ } \lhead[\thepage]{} }                  

\begin{center}
\vspace*{2pt}
{\Large \textbf{The structure of idempotent translatable quasigroups}}\\[30pt]
 {\large \textsf{\emph{Wieslaw A. Dudek \ and \ Robert A. R. Monzo}}}\\[30pt]
\end{center}
\date{}
{\footnotesize\textbf{Abstract.} We prove the main result that a groupoid of order $n$ is an idempotent and $k$-translatable quasigroup if and only if its multiplication is given by $x\cdot y=(ax+by)({\rm mod}\,n)$, where $a+b=1({\rm mod}\,n)$, $a+bk=0({\rm mod}\,n)$ and $(k,n)=1$. We describe the structure of various types of idempotent $k$-translatable quasigroups, some of which are connected with affine geometry and combinatorial algebra, and their parastrophes. We prove that such parastrophes are also idempotent $k$-translatable quasigroups and determine when they are of the same type as the original quasigroup. In addition, we find several different necessary and sufficient conditions making a $k$-translatable quasigroup quadratical.
 }

%%%%%%%%%%%%%%%%%%%%%%%%%%%%%%%%%%%%%%%%%%%%%%%%%%%%%%%%%%%%%%%%%%%%%%%%%%%%%%%%%%%%%%%%%

\section{Introduction}

Many collections of algebraic objects are {\em equivalent} to others. That is, there is a bijection between the collections $\bf{V}$ and $\bf{V}'$, which lets us move back and forth freely between them. When this is the case, we write $\bf{V}\equiv \bf{V}'$. This equivalence may hold only up to isomorphism. In this case we write $\bf{V}\approx\bf{V}'$.

For example, if {\bf G}, {\bf I}, {\bf H}, {\bf WQ}, {\bf GHI}, {\bf GH}, {\bf EBM} and {\bf NWQ} are the collections of groups, inverse semigroups, heaps, Ward quasigroups, generalized heaps that appear as the standard ternary operation of an inverse semigroup, generalized heaps, equivalence bimodules and the natural ternary operations of Ward quasigroups, then ${\bf G}\equiv{\bf WQ}$,  ${\bf G}\approx {\bf H}$, ${\bf G}\equiv{\bf NWQ}$, ${\bf NWQ}\approx{\bf H}$, ${\bf H}\approx{\bf WQ}$, ${\bf GHI}\equiv{\bf I}$ and ${\bf GH}\approx{\bf EBM}$ (cf. \cite{6,7,10}).

One hope might be that one could use results in, or properties of, one algebraic system to assist in proving important results that are difficult to prove if we remain only within an equivalent ($\equiv$  or $\approx$) algebraic system. It would be useful also if results in one system implied interesting results in the equivalent system.

In this paper, we consider the collection $\bf{QQ}$ of all quadruples $(Q,\cdot,L_s,R_s)$, where $(Q,\cdot)$ is a quadratical quasigroup with a pair $(L_s,R_s)$ of left and right translations respectively for some $s\in Q$, and the collection $\bf{A}$ of all quadruples $(Q,+,\lambda,\rho)$, where $(Q,+)$ is a commutative group with a pair of (commuting) automorphisms $(\lambda,\rho)$ such that $\lambda x+\rho x=x$ and $\rho\lambda x+\rho\lambda x=x$ for all $x\in Q$. Note that from $\lambda x+\rho x=x$ it follows that $\lambda$ is an automorphism of $(Q,+)$ if and only if $\rho$ is an automorphism of $(Q,+)$. In this case $\lambda\rho=\rho\lambda$. We prove that ${\bf QQ}\equiv{\bf A}$.

One major result of this paper is that a quadratical quasigroup of order $n$ is $k$-translatable if and only if it is induced by the additive group of integers modulo $n=4t+1$, where $t$ is some positive integer (Theorem \ref{T-36}). This is proved using the equivalence between $\bf{QQ}$ and $\bf{A}$; in fact, the proof jumps back and forth between these equivalent collections. Moreover, in Section $3$ we find necessary and sufficient conditions on a $k$-translatable quasigroup for it to be quadratical (Theorems \ref{T-32} and \ref{T-45}).

   In Section $4$ we prove the main result that a groupoid of order $n$ is an idempotent $k$-translatable quasigroup if and only if its multiplication is given by $x\cdot y=(ax+by)({\rm mod}\,n)$, $a+b=1({\rm mod}\,n)$, $a+bk=0({\rm mod}\,n)$ and $(k,n)=1$ (Lemma \ref{L-41}). Theorem \ref{T-425} determines the structure of $k$-translatable quadratical, hexagonal, $GS$, right modular, left modular, Stein, $ARO$ and $C3$ quasigroups. The structure of the parastrophes of such quasigroups is determined in Theorem \ref{T-51}. In Theorem \ref{T-trans} we prove that these parastrophes are all idempotent translatable quasigroups and find the value of translatability. In Table $3$ we determine when a parastrophe of an idempotent $k$-translatable quadratical (hexagonal, $GS$, $ARO$, $C3$, right modular or Stein) quasigroup is quadratical (hexagonal, $GS$, $ARO$, $C3$, right modular or Stein). In Table $4$ we find necessary and sufficient conditions for a parastrophe to be quadratical (hexagonal, $GS$, $ARO$, $C3$, right modular or Stein).

\section{Preliminary definitions and results}  %%%%%%%%%%%%%%%%%%%%%%%%%%%%%%%%%%%%%%%%%%%%%%%%%%%%%%%%%%%       2

Recall that a groupoid $(Q,\cdot)$ has {\it property A} if it satisfies the identity $xy \cdot x = zx \cdot yz$ \cite{3, 9}. 
It is called {\it right solvable $($left solvable$)$} if for any $\{a,b\} \subseteq Q$ there exists a unique $x\in G$ such that $ax = b$ ($xa = b$). It is {\it left $($right$)$ cancellative} if $xy = xz$ implies $y = z$ ($yx = zx$ implies $y = z$). It is a {\it quasigroup} if it is left and right solvable. 

Note that right solvable groupoids are left cancellative, left solvable groupoids are right cancellative and quasigroups are cancellative.

Volenec \cite{9} defined {\it quadratical groupoids} as right solvable groupoids satisfying property $A$ and proved some basic properties of these groupoids. Below, we list several such properties. Throughout the remainder of this paper we will use, without mention, the fact that quadratical groupoids are quasigroups that satisfy these properties.

\begin{theorem}\label{T-21} A quadratical groupoid is left solvable and satisfies the following identities:

$\begin{array}{rlll}
(1)&x = {x^2}, & (idempotency)\\
(2)&x \cdot yx = xy \cdot x , &(elasticity, \ flexibility)\\
(3)&x \cdot yx = xy \cdot x=yx \cdot y, &(strong \ elasticity)\\
(4)&yx \cdot xy = x, &(bookend)\\
(5)&x \cdot yz = xy \cdot xz, &(left \ distributivity)\\
(6)&xy \cdot z = xz \cdot yz, &(right \ distributivity)\\
(7)&xy \cdot zw = xz \cdot yw, &(mediality)\\
(8)&x(y \cdot yx) = (xy \cdot x)y, &\\
(9)&(xy \cdot y)x = y(x \cdot yx), &\\
(10)&xy = zw\,\Longleftrightarrow\,yz = wx, \ \ \ \ &(alterability).
\end{array}$
\end{theorem}

\begin{definition}\label{D-22}
{\bf QQ} is defined as the collection of quadruples $(Q,\cdot,\lambda,\rho)$, where $(Q,\cdot)$ is a quadratical quasigroup with commuting automorphisms $\lambda$ and $\rho$ satisfying for all $x,y,z\in Q$ and some $w\in Q$ the following conditions:

\medskip
$\begin{array}{rlll}
(11)&xy\cdot \lambda z=\rho x\cdot yz,\\[2pt]    
(12)&\lambda x\cdot\rho x=x,\\[2pt]    
(13)&\rho^{-1}x\cdot\lambda^{-1}y=\rho^{-1}y\cdot\lambda^{-1}x,\\[2pt]   
(14)&\rho^{-1}x\cdot\lambda^{-1}w=y .  
\end{array}$
\end{definition}

\begin{proposition}\label{P-23} If $(Q,\cdot,\lambda,\rho)\in\bf{QQ}$, then $\lambda=L_{\rho x\cdot\lambda x}$, $\rho=R_{\rho x\cdot\lambda x}$ and 
$\rho x\cdot\lambda x=\rho y\cdot\lambda y$ for all $x,y\in Q$. 
\end{proposition}
\begin{proof} For all $x,y\in Q$, by (13), $\rho^{-1}x\cdot\lambda^{-1}y=\rho^{-1}y\cdot\lambda^{-1}x$. Hence, $\rho\lambda(\rho^{-1} x\cdot\lambda^{-1} y)=\rho\lambda(\rho^{-1} y\cdot\lambda^{-1} x)$. Since $\rho\lambda=\lambda\rho$, $\lambda\rho^{-1}=\rho^{-1}\lambda$, so $\lambda x\cdot\rho y=\lambda y\cdot\rho x$. By alterability, $\rho y\cdot\lambda y=\rho x\cdot \lambda x$. By (12), $\lambda x\cdot\rho x=x$. But $(Q,\cdot)$ is bookend and so $\lambda x=(\rho x\cdot\lambda x)\cdot(\lambda x\cdot\rho x)=(\rho x\cdot\lambda x)\cdot x$ and $\rho x=(\lambda x\cdot\rho x)\cdot(\rho x\cdot\lambda x)=x\cdot (\rho x\cdot\lambda x)$. Therefore, $\rho=R_{\rho x\cdot \lambda x}$ and $\lambda=L_{\rho x\cdot\lambda x}$.  
\end{proof}

\begin{proposition}\label{P-24} If $(Q,\cdot)$ is a quadratical quasigroup, then $(Q,\cdot,L_s,R_s)\in \bf{QQ}$ for all $s\in Q$.
\end{proposition}
\begin{proof} Right and left distributivity imply that $R_s$ and $L_s$ are endomorphisms for each $s\in Q$. Thus, they are automorphisms of $(Q,\cdot)$. Elasticity implies that $R_s$ and $L_s$ commute. By mediality, $xy\cdot L_sz=R_sx\cdot yz$, so $(11)$ is valid. Bookend implies $x=L_sx\cdot R_sx$ , so $(12)$ is valid. Mediality also implies $L_sx\cdot R_sy=L_sy\cdot R_sx$. Then, since $R_sL_s=L_sR_s$ implies $R_sL_s^{-1}=L_s^{-1}R_s$, $\,L_sR_s(R_s^{-1}x\cdot L_s^{-1}y)=L_sx\cdot R_sy=L_sy\cdot R_sx=L_sR_s(R_s^{-1}y\cdot L_s^{-1}x)$. Hence $R_s^{-1}x\cdot L_s^{-1}y=R_s^{-1}y\cdot L_s^{-1}x$ and $(13)$ is valid. Finally, right solvability implies that for all $x,y\in Q$ there exists $u\in Q$ such that $R_s^{-1}x\cdot u=y$. Since $L_s^{-1}$ is an automorphism, there exists $w\in Q$ such that $R_s^{-1}x\cdot L_s^{-1}w=y$  and $(14)$ is valid. By definition then, $(Q,\cdot,L_s,R_s)\in \bf{QQ}$.  
\end{proof}

\begin{definition}\label{D-25}  ${\bf A}$ is defined as the collection of quadruples $(A,+,\lambda,\rho)$, where $(A,+)$ is a $2$-divisible commutative group with automorphisms $\lambda$ and $\rho$ such that:

\medskip
$\begin{array}{rlll}
(15)&\rho x+\lambda x=x,\\[2pt] 
(16)&\lambda\rho x+\lambda\rho x=x \\[2pt] 
\end{array}$

\noindent
for all $x\in A$.
\end{definition}

Note that $(15)$ and $\lambda,\rho$ automorphisms imply $\rho\lambda x=\lambda x-\lambda^2x=\lambda\rho x$, so $\rho\lambda=\lambda\rho$.

\begin{lemma}\label{L-26}
If $(A,+,\lambda,\rho)\in{\bf A}$, then for each $x\in A$ there exists only one $\overline{x}\in A$ such that $x=\overline{x}+\overline{x}$.
\end{lemma}
\begin{proof}
Indeed, by $(16)$, for $x=y+y$ we have $\lambda\rho x=\lambda\rho(y+y)=\lambda\rho y+\lambda\rho y=y$. 
\end{proof}

\begin{definition}\label{D-27} For $(A,+,\lambda,\rho)\in {\bf A}$ we define new product on $A$ by putting:

\smallskip
$(17)$ \ $x\oplus y=\rho x+\lambda y$ \ \ for all $x,y\in A$. 

\smallskip\noindent
Similarly, for $(Q,\cdot,\lambda,\rho)\in{\bf QQ}$ we define:

\smallskip 
$(18)$ \ $x\odot y=\rho^{-1}x\cdot\lambda^{-1}y$ \ \ for all $x,y\in Q$.
\end{definition}
Then $(A,\oplus)$ and $(Q,\odot)$ are quasigroups.

\begin{proposition}\label{P-210} Suppose that $(A,+,\lambda,\rho)\in {\bf A}$  and $(Q,\cdot,\lambda,\rho)\in \bf{QQ}$. Then 

\medskip
$\begin{array}{rll}
(19)&(x\oplus y)+(z\oplus w)=(x+z)\oplus (y+w) \ \ for \ x,y,z,w\in A, \ and \\[2pt]
(20)&(x\cdot y)\odot(z\cdot w)=(x\odot z)\cdot (y\odot w) \ \ for \ x,y,z,w\in Q.
\end{array}$
\end{proposition}
\begin{proof} We have $(x\oplus y)+(z\oplus w)=(\rho x+\lambda y)+(\rho z+\lambda w)=\rho(x+z)+\lambda(y+w)=(x+z)\oplus (y+w)$ and 
$(x\cdot y)\odot (z\cdot w)=\rho^{-1}(x\cdot y)\cdot\lambda^{-1}(z\cdot w)=(\rho^{-1}x\cdot \lambda^{-1}z)\cdot (\rho^{-1}y\cdot\lambda^{-1}w)=(x\odot z)\cdot (y\odot w)$. 
\end{proof}

The following Lemma follows from the proof of the main Theorem in \cite{2}.

\begin{lemma}\label{L-29}
If $(A,+,\lambda,\rho)\in{\bf A}$, then $\lambda\rho=\rho\lambda$, $(A,\oplus)$ is a quadratical quasigroup and $(x\oplus y)\oplus z= (z\oplus x)\oplus (y\oplus z)$ for all $x,y,z\in A$.
\end{lemma}
\begin{theorem}\label{T-211} If $(A,+,\lambda,\rho)\in{\bf A}$, then the following statements are equivalent:    
  
\medskip
$\begin{array}{rll}
(21)&x\oplus y=z\oplus w,\\[2pt]
(22)&x+(w\oplus y)=z+(y\oplus w) \ \ and \ \ y+(x\oplus z)=w+(z\oplus x). 
\end{array}$
\end{theorem}
\begin{proof} $(21)\Rightarrow (23)$. Since, by Lemma \ref{L-29}, $(A,\oplus)$ is quadratical, it is alterable. Hence, $y\oplus z=w\oplus x$. By definition, $x\oplus y=z\oplus w$ and $y\oplus z=w\oplus x$ imply that
$\rho x+\lambda y=\rho z+\lambda w$ \ and \ $\rho y+\lambda z=\rho w+\lambda x$. So, $(\rho x+\lambda y)+(\rho w+\lambda x)=(\rho z+\lambda w)+(\rho y+\lambda z)$ 
and $(\rho x+\lambda y)+(\rho y+\lambda z)=(\rho z+\lambda w)+(\rho w+\lambda x)$. This proves $(22)$.

$(22)\Rightarrow (21)$. By definition,

\medskip
$\begin{array}{rlll}
(23)&(\rho x+\lambda x)+(\rho w+\lambda y)=(\rho z+\lambda z)+(\rho y+\lambda w),\\[4pt]
(24)&(\rho y+\lambda y)+(\rho x+\lambda z)=(\rho w+\lambda w)+(\rho z+\lambda x).
\end{array}
$

\medskip\noindent
Then, 
$$
\rho y+\lambda z=(\rho z+\lambda w)+(\rho w+\lambda x)-(\rho x+\lambda y)=(\rho x+\lambda y)+(\rho w+\lambda x)-(\rho z+\lambda w).
$$
Therefore, $2(\rho x+\lambda y)=2(\rho z+\lambda w)$, which, by Lemma \ref{L-26}, means that $\rho x+\lambda y=\rho z+\lambda w$. This proves $(21)$. 
 \end{proof}

\begin{theorem}\label{T-214} 
${\bf A}\equiv{\bf QQ}$.
\end{theorem}
\begin{proof} We prove that the mappings $\Psi:{\bf A}\rightarrow{\bf QQ}$; $(A,+,\lambda,\rho)\mapsto (A,\oplus,\lambda,\rho)$ and $\Phi:{\bf QQ}\rightarrow{\bf A}$; $(Q,\cdot,\lambda,\rho)\rightarrow (Q,\odot,\lambda,\rho)$ are mutually inverse mappings.

Let $(A,+,\lambda,\rho)\in{\bf A}$. Then, by Lemma \ref{L-29}, $(A,\oplus)$ is a quadratical quasigroup. Moreover, $(x\oplus y)\oplus\lambda z=(\rho^2x+\rho\lambda y)+\lambda^2z=\rho^2x+\lambda(\rho y+\lambda z)=\rho x\oplus(y\oplus z)$. So, $(11)$ holds (for $\oplus$). Also, $\lambda x\oplus\rho x=\rho\lambda x+\lambda\rho x=\lambda\rho x+\lambda\rho x=x$, which proves $(12)$. Then, $\rho^{-1}x\oplus\lambda^{-1}y=x+y=y+x=\rho^{-1}y\oplus\lambda^{-1}x$ and so $(13)$ holds. Since $y=x+(y-x)$, $y=\rho^{-1}x\oplus\lambda^{-1}w$ for $w=y-x$. This proves $(14)$. Finally, $\lambda(x\oplus y)=\lambda(\rho x+\lambda y)=\rho\lambda x+\lambda^2 y=\lambda x\oplus\lambda y$ and so $\lambda$ is an automorphism of $(A,\oplus)$ and, similarly, so is $\rho$. Hence, $(A,\oplus,\lambda,\rho)\in {\bf QQ}$ and $\Psi$ is well-defined. 

Conversely, let $(Q,\cdot,\lambda,\rho)\in{\bf QQ}$. Then $\lambda\rho=\rho\lambda$ and $\rho^{-1}\lambda^{-1}=\lambda^{-1}\rho^{-1}$. Now $(11)$, for $\,x=\rho^{-2}a$, $\,y=\lambda^{-1}\rho^{-1}b\,$ and $\,z=\lambda^{-2}c$, gives $(\rho^{-2}a\cdot\rho^{-1}\lambda^{-1}b)\cdot\lambda^{-1}c=\rho^{-1}a\cdot (\lambda^{-1}\rho^{-1}b\cdot\lambda^{-2}c)$. Hence, 
$(a\odot b)\odot c=\rho^{-1}(\rho^{-1}a\cdot\lambda^{-1}b)\cdot\lambda^{-1}c=(\rho^{-2}a\cdot\rho^{-1}\lambda^{-1}b)\cdot\lambda^{-1}c=\rho^{-1}a\cdot (\lambda^{-1}\rho^{-1}b\cdot\lambda^{-2}c)=a\odot (b\odot c)$ and so $(Q,\odot)$ is a semigroup. By, $(13)$, it is commutative. From $(14)$ it follows that for all $x,y\in Q$ there exists $w\in Q$ such that $y=x\odot w$. Hence, $(Q,\odot)$ is a commutative group.  

Now since $\rho x\odot \lambda x=x\cdot x=x$, $(15)$ holds. Also, $\rho(x\odot y)=\rho(\rho^{-1}x\cdot\lambda^{-1}y)=x\cdot \lambda^{-1}\rho y=\rho x\odot\rho y$. So $\rho$ is an automorphism of $(Q,\odot)$. Similarly, $\lambda$ is an automorphism of $(Q,\odot)$. Finally, by $(12)$, $\lambda\rho x\odot\lambda\rho x=\rho\lambda x\odot\lambda\rho x=\rho^{-1}(\rho\lambda x)\cdot\lambda^{-1}(\lambda\rho x)=\lambda x\cdot\rho x=x$ and so $(16)$ holds too. Consequently, $(Q,\odot,\lambda,\rho)\in {\bf A}$ and $\Phi$ is well-defined. Clearly, $\Psi$ and $\Phi$ are mutually inverse mappings. 
\end{proof}

	%%%%%%%%%%%%%%%%%%%%%%%%%%%%%%%%%%%%%%%%%%%%%%%%%%%%%%%%%%%%%%%%          3

\section{$k$-translatable quasigroups} 			% 3							

All groupoids considered in this section are finite. For simplicity we assume that they have form $Q=\{1,2,\ldots,n\}$ with the {\it natural ordering} $1,2,\ldots,n$, which is always possible by renumeration of elements. Moreover, instead of $i\equiv j({\rm mod}\,n)$ we will write $[i]_n=[j]_n$. Additionally, in calculations of modulo $n$, we assume that $0=n$. 

Recall that a groupoid $(Q,\cdot)$ is {\em $k$-translatable} if and only if for all $i,j\in Q$ we have $i\cdot j=[i+1]_n\cdot [j+k]_n$, or equivalently, $i\cdot j=a_{[k-ki+j]_n}$, where $a_1,a_2,\ldots,a_n$ is the first row of the multiplication table of this groupoid. Then the sequence $a_1,a_2,\ldots,a_n$ is called a {\it $k$-translatable sequence}.

For the original definition of $k$-translatability see \cite{4} or \cite{5}.

Note that a groupoid may be $k$-translatable for one ordering but not for another, i.e., $a_1,a_2,\ldots,a_n$ may be a $k$-translatable sequence for the ordering $1,2,\ldots,n$ but in this groupoid for the ordering $2,1,3,\ldots,n$ may not be any $k$-translatable sequence. However the following result is true \cite[Lemma 2.7]{5}.

\begin{lemma}\label{L-31}
The sequence $a_1,a_2,\ldots,a_n$ is a $k$-translatable sequence of $(Q,\cdot)$ with respect to the ordering $1,2,\ldots,n$ if and only if $a_k,a_{k+1},\ldots,a_n,a_1,\ldots,a_{k-1}$ is a $k$-translatable sequence of $(Q,\cdot)$ with respect to the ordering $n,1,2,\ldots,n-1$.
\end{lemma}

If $(Q,\cdot)$ is a groupoid, then $(Q,*)$, where $x*y=y\cdot x$ is called the {\it dual groupoid}. It is clear that $(Q,*)$ is idempotent (medial) if and only if $(Q,\cdot)$ is idempotent (medial). Moreover, as a consequence of Theorem 4.1 in \cite{5} we obtain the following result.

\begin{proposition}\label{dual}
The dual groupoid of a left cancellative $k$-translatable groupoid of order $n$ is $k^*$-translatable if and only if $[kk^*]_n=1$.
\end{proposition}

Below we present several characterizations of finite $k$-translatable quadratical quasigroups.

\begin{theorem}\label{T-32} Let $(Q,\cdot)$ be a $k$-translatable quasigroup of order $n$. Then the following statements are equivalent:

\medskip
$\begin{array}{rlll}
(a)&  (Q,\cdot)\;is\; quadratical,\\[2pt]
(b)&  [i+k(z\cdot i)]_n=[(j\cdot z)+k(i\cdot j)]_n \ \ for\; all\;\; i,j,z\in Q,\\[2pt]
(c)&  [(z\cdot i)+k(j\cdot z)]_n=[ki+(i\cdot j)]_n \ \ and \ \ [k^2]_n=[-1]_n,\\[2pt]
(d)&  Q\;is\; idempotent,\; medial\; and \; [k^2]_n=[-1]_n .
\end{array}$
\end{theorem}
\begin{proof} $(a)\Leftrightarrow (b)$. In a $k$-translatable quasigroup $(Q,\cdot)$ we have $i\cdot j=a_{[k-ki+j]_n}$ and $a_i=a_j$ if and only if $i=j$. So, the condition $A$:
$(i\cdot j)\cdot i=(z\cdot i)\cdot (j\cdot z)$ defining a quadratical quasigroup is equivalent to $a_{[k-k(i\cdot j)+i]_n}=a_{[k-k(z\cdot i)+(j\cdot z)]_n}$, i.e., to $[i+k(z\cdot i)]_n=[(j\cdot z)+k(i\cdot j)]_n$.

$(a)\Leftrightarrow (c)$. A quadratical quasigroup is alterable (Theorem \ref{T-21}). By Corollary 2.18 in \cite{5},a $k$-translatable left cancellative groupoid (and so a quasigroup too) is alterable if and only if $[k^2]_n=[-1]_n$. The condition $A$ defining  a quadratical quasigroup is, as in the previous part of this proof, equivalent to $[i+k(z\cdot i)]_n=[(j\cdot z)+k(i\cdot j)]_n$. Multiplying both sides of this equation by $k$ we obtain $(c)$, and conversely, multiplying $(c)$ by $k$ we obtain $(b)$. This proves the equivalence $(a)$ and $(c)$.

$(a)\Leftrightarrow (d)$. By Theorem \ref{T-21}, a quadratical quasigroup is idempotent, medial and alterable. If it is $k$-translatable, then $[k^2]_n=[-1]_n$, by Corollary 2.18 in \cite{5}.

Conversely, by Corollary 2.18 in \cite{5}, a $k$-translatable quasigroup with the property $[k^2]_n=[-1]_n$ is alterable. Since it is idempotent and medial, it is elastic. Idempotency, elasticity and alterability imply bookend. By Theorem 2.30 in \cite{3}, a bookend, idempotent and medial quasigroup is quadratical.
\end{proof}

\begin{corollary}\label{C-43} 
A $k$-translatable, left cancellative, right distributive groupoid of order $n$ is a quadratical quasigroup if and only if $[k^2]_n=[-1]_n$.
\end{corollary}
\begin{proof} A quadratical groupoid is an alterable quasigroup. If it is a $k$-translatable, then $[k^2]_n=[-1]_n$, by Corollary 2.18 in \cite{5}.

Conversely, a right distributive and left cancellative groupoid is idempotent and elastic. If it is $k$-translatable, then, by Proposition 2.13 from \cite{5}, it is left distributive too, and $[k^2]_n=[-1]_n$ shows, by Corollary 2.18 in \cite{5}, that it is alterable. An alterable, idempotent, elastic groupoid satisfies $j=j\cdot j=(i\cdot j)\cdot (j\cdot i)$, so it is bookend.  An alterable, left distributive, bookend groupoid is a quadratical quasigroup \cite[Theorem 2.30]{3}. 
\end{proof}

\begin{corollary} \label{C-44}
A $k$-translatable, right solvable and right distributive, alterable groupoid of order $n$ is a quadratical quasigroup and $[k^2]_n=[-1]_n$.
\end{corollary}
\begin{proof} By Proposition 2.9 from \cite{5} such groupoid is an idempotent quasigroup. Since it is alterable, $[k^2]_n=[-1]_n$. By Corollary \ref{C-43} it is quadratical. 
\end{proof}

\begin{theorem} \label{T-45} Let $(Q,\cdot)$ be a $k$-translatable, quasigroup of order $n$. Then the following statements are equivalent:
\begin{enumerate}
\item[$(a)$] \ $(Q,\cdot)$ is quadratical,
\item[$(b)$] \ $(Q,\cdot)$ is right distributive and $[k^2]_n=[-1]_n$,
\item[$(c)$] \ $(Q,\cdot)$ is right distributive and alterable, 
\item[$(d)$] \ $(Q,\cdot)$ is left distributive and alterable.
\end{enumerate}
\end{theorem}
\begin{proof} A $k$-translatable quadratical quasigroup of order $n$ satisfies $(b)$, $(c)$ and $(d)$. Hence $(a)$ implies $(b)$, $(c)$ and $(d)$. By Corollary \ref{C-43}, $(b)$ implies $(a)$. By Corollary \ref{C-44}, $(c)$ implies $(a)$. Thus $(a)$, $(b)$ and $(c)$ are equivalent.

A $k$-translatable quasigroup satisfying $(d)$ is idempotent. Hence, it is right distributive \cite[Proposition 2.13]{5}. Thus $(d)$ implies $(c)$, and in the consequence $(a)$. This completes the proof.  
\end{proof}

Finally, we will describe $k$-translatable quasigroups induced by the additive group $\mathbb{Z}_n$ of positive integers modulo $n$.

As is known (cf. \cite{2}), all quadratical quasigroups are uniquely determined by some commutative groups and their two commuting automorphisms. In particular, a quadratical quasigroup induced by the additive group $\mathbb{Z}_n$ of positive integers modulo $n$ has the form $x\cdot y=[ax+(1-a)y]_n$, where $a\in \mathbb{Z}_n$ and $[2a^2-2a+1]_n=0$ \cite[Theorem 4.8]{3}. For a fixed $a\in\mathbb{Z}_n$ such a quasigroup may be $k$-translatable for only one value of $k$ \cite[Theorem 9.3]{4}. This is possible only for $[k^2]_n=[-1]_n$ \cite[Corollary 2.18]{5}. Moreover, a quadratical quasigroup induced by $\mathbb{Z}_n$ is $k$-translatable if and only if its dual quasigroup $(n-k)$-translatable \cite[Theorem 9.4]{4}.

\begin{theorem}\label{T-36} A quadratical quasigroup of order $n$ is $k$-translatable if and only if it is induced by the additive group $\mathbb{Z}_n$. 
\end{theorem}

\begin{proof} {\sc Necessity}. Let $(Q,\cdot)$ be a $k$-translatable quadratical quasigroup of order $n$. Then $x\cdot y=\rho x+\lambda y$ for some commutative group $(Q,+)$ of order $n$ and two of its commuting automorphisms $\rho$ and $\lambda$. 

Let $e$ be the neutral element of $(Q,+)$. Then 
$$
e=\rho e+\lambda e=e\cdot e=[e+1]_n\cdot [e+k]_n=\rho[e+1]_n+\lambda[e+k]_n
$$ 
because $(Q,\cdot)$ is $k$-translatable. Using Lemma \ref{L-31} repeatedly, we can choose an ordering of $Q$ such that $e=n$. Hence, $e=n=n\cdot n=1\cdot k=\rho 1+\lambda k$, which  means that in the group $(Q,+)$ we have $\lambda k=-\rho 1$. 

Now, we prove by induction on $i$ that $\rho i=[i(\rho 1)]_n$ for every $i\in Q$.  

Clearly, $\rho 1=1(\rho 1)$. Assume $\rho j=j(\rho 1)$ for all $j\leq i-1$.  
Then, $[(i-1)(\rho 1)]_n=\rho(i-1)=\rho(i-1)+\lambda e=(i-1)\cdot n=i\cdot k=\rho i+\lambda k=\rho i-\rho 1$, which implies $\rho i=[i(\rho 1)]_n$ for all $i\in Q$. In particular, $e=n=\rho n=n(\rho 1)$.

Define $\varphi\colon (Q,+)\rightarrow (\mathbb{Z}_n,+)$ as $\varphi(\rho i)=i$. Then, $\varphi$ is one-to-one and onto. Also, $\varphi(\rho i+\rho j)=\varphi([i(\rho 1)]_n+[j(\rho 1)]_n)=\varphi([i+j]_n(\rho 1))=\varphi(\rho[i+j]_n)=[i+j]_n=[\varphi(\rho i)+\varphi(\rho j)]_n$. So, $(Q,+)$ and $\mathbb{Z}_n$ are isomorphic. Hence, $(Q,\cdot)$ is induced by the additive group $\mathbb{Z}_n$.

{\sc Sufficiency}. A quadratical quasigroup induced by the additive group $\mathbb{Z}_n$ has the form $x\cdot y=[ax+(1-a)y]_n$, where $a\in \mathbb{Z}_n$ and $[2a^2-2a+1]_n=0$ \cite[Theorem 4.8]{3}. Since $[a+(1-2a)(1-a)]_n=[1-2a+2a^2]_n=0$, this quasigroup is $[1-2a]_n$-translatable \cite[Lemma 9.1]{4}. 
 \end{proof}

As a simple consequence of the above theorem and Proposition 3.4 in \cite{3} we obtain the following two corollaries.

\begin{corollary}\label{C-32} A quadratical quasigroup of order $n$ is $k$-translatable if and only if it is induced by the group $\mathbb{Z}_{4t+1}$, for some positive integer $t$.
\end{corollary}

\begin{corollary}\label{C-33} The set of all orders of $k$-translatable quadratical quasigroups is $\{n: n=p_1^{\alpha_1}p_2^{\alpha_2}\cdots p_t^{\alpha_t}\}$, where the $p_i$ are different primes such that $p_i\equiv 1({\rm mod}\,4)$ for all $i=1,2,\ldots,t$.
\end{corollary}

																				%%%%%%%%%%%%%%%%%%%%%%%%%%%%%%%%%%%%%%%%%%%%%%%%%%%%%%%%%%%%%%%%          4
	
\section{Idempotent $k$-translatable groupoids} 

Let's recall that an idempotent $k$-translatable groupoid of order $n$ is left cancellative \cite[Lemma 2.10]{5}. Thus in a $k$-translatable sequence $a_1,a_2,\ldots,a_n$ of this groupoid all elements are different, i.e., $a_i=a_j$ if and only for $i=j$. Such a groupoid may not be right cancellative, but if $(k,n)=1$, then it is a quasigroup \cite[Theorem 2.14]{5}.
For every odd $n$ and every $k>1$ such that $(k,n)=1$ there is only one (up to isomorphism) idempotent $k$-translatable quasigroup \cite[Theorem 2.12]{5}.
For even $n$ there are no such quasigroups \cite[Theorem 8.9]{4}.

\begin{lemma}\label{L-41} In an idempotent $k$-translatable groupoid $(Q,\cdot)$ with a $k$-translatable sequence $a_1,a_2,\ldots,a_n$, for $c=2\cdot 1$ and $x,y\in Q$ 
\begin{enumerate}
\item[$(i)$] \ $[kc-c-2k+1]_n=0$,  
\item[$(ii)$] \ $[(1-k)x]_n=[(1-k)y]_n$ implies $x=y$,                                                                                                         
\item[$(iii)$] \ $a_{[i+1]_n}=[a_i+(2-c)]_n=[1+i(2-c)]_n$,
 \item[$(iv)$] \ $x\cdot y=[(c-1)x+(2-c)y]_n$,
\item[$(v)$] \ $(Q,\cdot)$ is a quasigroup if and only if $(k,n)=1$,
\item[$(vi)$] \ if $(Q,\cdot)$ is a quasigroup, then $n$ is odd.
\end{enumerate}
	\end{lemma}
\begin{proof} 
$(i)$. Since $(Q,\cdot)$ is idempotent and $k$-translatable, $c\cdot c=c=2\cdot 1=1\cdot [1-k]_n=[1+(c-1)]_n\cdot [(1-k)+k(x-1)]_n=c\cdot [1-k+kc-k]_n$. This, by left cancellativity,  implies $c=[1-2k+kc]_n$. 

$(ii)$. $[(1-k)x]_n=[(1-k)y]_n$ gives $[x-y]_n=[k(x-y)]_n$, i.e., $t=[tk]_n$ for $t=[x-y]_n$. Then, by $(i)$, we have $0=[t(kc-c-2k+1)]_n=[-t]_n=[y-x]_n$. Hence $x=y$.

$(iii)$. Since $a_i=a_i\cdot a_i=a_{[k-ka_i+a_i]_n}$, by left cancellativity $i=[k-ka_i+a_i]_n$. Thus, $[i-k]_n=[(1-k)a_i]_n$.
Also $[(1-k)(2-c)]_n=[2-c-2k+kc]_n=1$, by $(i)$. Now, using these two facts, we obtain $[(1-i)(a_i+(2-c))]_n=[(1-k)a_i+1]_n=[i-k+1]_n=[(1-k)a_{i+1}]_n$. 
From this, by $(ii)$, we deduce that $a_{[i+1]_n}=[a_i+(2-c)]_n$. But $a_1=1\cdot 1=1$, so $a_2=[1+(2-c)]_n$, and, by induction, $a_{[i+1]_n}=[1+i(2-c)]_n$.

$(iv)$. We have $x\cdot y=a_q$, where $q=[k-kx+y]_n$. This, by $(iii)$, gives

$x\cdot y=a_q=[1+(k-kx+y-1)(2-c)]_n=[1+(2-c)(k-1)-k(2-c)x+(2-c)y]_n$

\hspace*{6mm}$\stackrel{(i)}{=}[(kc-2k)x+(2-c)y]_n\stackrel{(i)}{=}[(c-1)x+(2-c)y]_n.$

$(v)$. This follows from Theorem 2.14 \cite{5} and Lemma 2.15 \cite{5}.

$(vi)$. Observe that $(i)$ can be rewritten in the form $[(k-1)(c-1)-k]_n=0$. This, for odd $k$, means that $n$ is odd. 
From $(k,n)=1$ it follows that $n$ must be odd also for even $k$. So, in both cases $n$ is odd.
\end{proof}

\begin{theorem}\label{T-42}
A naturally ordered groupoid $(Q,\cdot)$ of order $n$ is idempotent and $k$-translatable if and only if $x\cdot y=[ax+by]_n$ for all $x,y\in Q$ and some $a,b\in \mathbb{Z}_n$ and

$(1)$ \ $[a+b]_n=1$, \ $[a+bk]_n=0$, or equivalently,

$(2)$ \ $1\cdot 1=1$, \ $1\cdot k=n$.
\end{theorem}
\begin{proof}
By Lemma \ref{L-41}$(iv)$, an idempotent $k$-translatable groupoid $(Q,\cdot)$ of order $n$ has the form $x\cdot y=[ax+by]_n$, where $a=[c-1]_n$, \ $b=[2-c]_n$ and $c=2\cdot 1$. Then obviously, $[a+b]_n=1$ and $x\cdot y=[x+1]_n\cdot [y+k]_n$, which gives $[a+bk]_n=0$. 

The converse statement is obvious. 
\end{proof} 

Note that in the above theorem $(b,n)=1$. Indeed, since $(Q,\cdot)$ is left cancellative, $x\cdot y=x\cdot w$ implies $y=w$. But this is possible only in the case when $(b,n)=1$.
So, $\psi(x)=[bx]_n$ is an automorphism of the additive group $\mathbb{Z}_n$, but $\varphi(x)=[ax]_n$ may not be one-to-one. For example, the groupoid with the ope\-ration $x\cdot y=[4x+5y]_8$ is idempotent and $4$-translatable, but it is not a quasigroup.
 
\begin{corollary}\label{C-div} If a naturally ordered idempotent groupoid of order $n$ has the form $x\cdot y=[ax+by]_n$, then it is $k$-translatable if and only if $[a+bk]_n=0$ and $(k,n)$ divides $(a,n)$.
\end{corollary}
\begin{proof}
Indeed, $(k,n)=d$ together with $[a+bk]_n=0$ gives $d|a$.
\end{proof}

\begin{corollary}\label{C-kn} 
If a naturally ordered idempotent quasigroup $(Q,\cdot)$ of order $n$ has the form $x\cdot y=[ax+by]_n$, then it is $k$-translatable if and only if $[a+bk]_n=0$ and $(k,n)=1$.
\end{corollary}

\begin{corollary} An idempotent $k$-translatable groupoid is medial.
\end{corollary}

\begin{corollary} \label{C-46}
A naturally ordered groupoid $(Q,\cdot)$ of order $n$ is a $k$-translatable quadratical quasigroup if and only if $x\cdot y=[ax+by]_n$ for all $x,y\in Q$ and some $a,b\in \mathbb{Z}_n$ such that

$(1)$ \ $[a+b]_n=1$, \ $[a+bk]_n=0$, \ $[2ab]_n=1$, or equivalently,

$(2)$ \ $1\cdot 1=1$, \ $1\cdot k=n$, \ $[a^2+b^2]_n=0$.      
\end{corollary}
\begin{proof}
Since $(Q,\cdot)$ is idempotent and $k$-translatable, by our Theorem \ref{T-42} and Theorem 4.8 in \cite{3}, the multiplication in such a quasigroup is given by $x\cdot y=[ax+by]_n$, where $a,b\in \mathbb{Z}_n$ are such that $[a+b]_n=1$ and $[2a^2-2a+1]_n=0$. The last implies $[2ab]_n=1$. By translatability we also have $[a+bk]_n=0$. 

Conversely, if a groupoid of order $n$ satisfies the above conditions, then $b=[1-a]_n$, which, by $[2ab]_n=1$, implies $[2a^2-2a+1]_n=0$ and $(a,n)=(b,n)=1$. So, this groupoid is quadratical \cite[Theorem 4.8]{3}. Since $(a,n)=(b,n)=1$, it is a quasigroup.
\end{proof}

Note that by \cite[Proposition 3.4]{3}, the order of a finite quadratical quasigroup is $n=4t+1$ for some $t=0,1,2,\ldots$

\begin{corollary}
A quadratical quasigroup of the form $x\cdot y=[ax+by]_n$ is $k$-translatable only for $k=[1-2a]_n$. Its dual quasigroup is $[2a-1]_n$-translatable.
\end{corollary}
\begin{proof}
In such a quasigroup $[2a^2-2a+1]_n=0$, $[2ab]_n=1$ and $[a+bk]_n=0$.
Multiplying the last equation by $2a$ we obtain $[2a^2+k]_n=0$, whence $k=[-2a^2]_n=[1-2a]_n$.
Since, $[(1-2a)(2a-1)]_n=1$, by Proposition \ref{dual}, the dual quasigroup is $[2a-1]_n$-translatable.
\end{proof}

Using Theorem \ref{T-42} we can determine the structure of various $k$-translatable quasigroups strongly connected with the affine geometry and combinatorial designs.

We will start with {\it hexagonal quasigroups}, which are connected with affine geometry (see for example \cite{8a}), i.e., with idempotent medial quasigroups satisfying the identity $x\cdot yx=y$.
By Lemma \ref{L-41}, $k$-translatable hexagonal quasigroups have odd order $n$ and $(k,n)=1$.

\begin{proposition}\label{P-48} A naturally ordered groupoid $(Q,\cdot)$ of order $n$ is a $k$-translatable hexagonal quasigroup if and only if $[k^2-k+1]_n=0$ and $x\cdot y=[(1-k)x+ky]_n$.
\end{proposition}
\begin{proof} {\sc Necessity}. Let $(Q,\cdot)$ be a naturally ordered $k$-translatable hexagonal quasigroup of order $n$. Since, by definition, $(Q,\cdot)$ is idempotent, by Lemma \ref{L-41}, we have $x\cdot y=[(c-1)x+(2-c)y]_n$, where $c=2\cdot 1$. From hexagonality we obtain $1\cdot c=1\cdot (2\cdot 1)=2$, which, by $k$-translatability, implies $1\cdot(c-1)=c$. Therefore, $(Q,\cdot)$ is $n-(c-2)$-translatable and so $k=[2-c]_n$. Consequently, $c=[2-k]_n$ and $x\cdot y=[(1-k)x+ky]_n$. Moreover, $0=[a+bk]_n=[1-k+k^2]_n$. 

{\sc Sufficiency}. It is easy to see that a naturally ordered groupoid $(Q,\cdot)$ of odd order $n$ with $x\cdot y=[(1-k)x+ky]_n$, where $[k^2-k+1]_n=0$, is idempotent, medial, $k$-translatable and satisfies the identity $x\cdot yx=y$. Since $0=[k^2-k+1]_n=[k(k-1)+1]_n$, we also have $(k,n)=(k-1,n)=1$. Thus, $(Q,\cdot)$ is a quasigroup.
\end{proof}

\begin{corollary} A naturally ordered groupoid $(Q,\cdot)$ of order $n$ is a $k$-translatable hexagonal quasigroup if and only if $x\cdot y=[(c-1)x+(2-c)y]_n$, where $c=2\cdot 1$, $k=[2-c]_n$ and $[c^2-3c+3]_n=0$.
\end{corollary}

\begin{corollary} A naturally ordered hexagonal quasigroup $(Q,\cdot)$ of order $n$ may be $k$-translatable only for $k=[2-c]_n$, where $c=2\cdot 1$.
\end{corollary}

\begin{corollary} 
A hexagonal quasigroup of order $n$ is $k$-translatable if and only if its dual quasigroup is $[1-k]_n$-translatable.
\end{corollary}
\begin{proof}
This is a consequence of Proposition \ref{dual} and fact that $1=[k-k^2]_n=[k(1-k)]_n$.
\end{proof}
\begin{corollary} There is $($up to isomorphism$)$ only one $k$-translatable commutative hexagonal quasigroup. It is induced by the group $\mathbb{Z}_3$ and has the form $x\cdot y=[2x+2y]_3$.
\end{corollary}
\begin{proof} Indeed, in this case $[1-k]_n=k$ and $k=n-1$. So, $[3]_n=0$. Thus, $n=3$ and consequently, $x\cdot y=[2x+2y]_3$.
\end{proof}

{\it $GS$-quasigroups $($golden section quasigroups$)$} defined in \cite{Vol} are used to describe various objects in affine geometry (see for example \cite{Kol} and \cite{Vol}). They are defined as idempotent quasigroups satisfying the (mutually equivalent) identities $x(xy\cdot z)\cdot z=y$ and $x\cdot (x\cdot yz)z=y$. It is not difficult to see that these quasigroups are medial \cite{Vol}. Hence, by Lemma \ref{L-41}, $k$-translatable $GS$-quasigroups have odd order $n$ and $(k,n)=1$.

\begin{proposition}\label{P-413} A naturally ordered groupoid $(Q,\cdot)$ of order $n$ is a $k$-transla\-ta\-ble $GS$-quasigroup if and only if $x\cdot y=[(k-1)x+(2-k)y]_n$ and $[k^2-3k+1]_n=0$.
\end{proposition}
\begin{proof} By Theorem \ref{T-42}, the multiplication in a $k$-translatable $GS$-quasigroup of order $n$ is defined by $x\cdot y=[ax+by]_n$, where $[a+b]_n=1$, $[a+bk]_n=0$ and $(a,n)=1$. From $x(xy\cdot z)\cdot z=y$ we obtain $[a^2-a-1]_n=0$ and $[ab+1]_n=0$. Multiplying $[a+bk]_n=0$ by $a$ we get $[a^2-k]_n=0$. So, $k=[a^2]_n=[a+1]_n$. Thus, $a=k-1$. Consequently $x\cdot y=[(k-1)x+(2-k)y]_n$ and $[k^2-3k+1]_n=0$.

Conversely, a groupoid satisfying these conditions is idempotent and $k$-transla\-ta\-ble. It also satisfies the identity $x(xy\cdot z)\cdot z=y$. From $[(k-1)(2-k)+1]_n=[-(k^2-3k+1)]_n=0$ it follows that $(k-1,n)=(2-k,n)=1$. Thus, this groupoid is a $GS$-quasigroup.
\end{proof}

\begin{corollary}
 A naturally ordered groupoid $(Q,\cdot)$ of order $n$ is a $k$-transla\-ta\-ble $GS$-quasigroup if and only if its dual groupoid is a $[3-k]_n$-translatable $GS$-quasigroup.
\end{corollary}

\begin{corollary}
A commutative $GS$-quasigroup is $k$-translatable only for $k=4$. It has the form $x\cdot y=[3x+3y]_5$.
\end{corollary}

\begin{proof}
Indeed, if it is commutative, then $[k-1]_n=[2-k]_n$, i.e., $[2k]_n=3$. From $2\cdot1=1\cdot 2$ it follows $k=n-1$. Thus $3=[2k]_n=[-2]_n$. So, $n=5$ and $k=4$.
\end{proof}

Other quasigroups associated with the affine geometry are $ARO$-quasigroups defined as idempotent medial quasigroups satisfying the identity  
$xy\cdot y=yx\cdot x$ (cf. \cite{11}). Thus, $k$-translatable $ARO$-quasigroups have odd order $n$ and $(k,n)=1$.

\begin{proposition}\label{P-416} A naturally ordered groupoid $(Q,\cdot)$ of order $n$ is a $k$-translatable $ARO$-quasigroup if and only if $x\cdot y=[ax+by]_n$ for some  $a,b\in\mathbb{Z}_n$ such that $[a+b]_n=1$, $[a+bk]_n=0$ and $[2a^2]_n=1$.   
\end{proposition}
\begin{proof} If $(Q,\cdot)$ is a $k$-translatable $ARO$-quasigroup, then, by Theorem \ref{T-42}, $x\cdot y=[ax+by]_n$, $[a+b]_n=1$ and $[a+bk]_n=0$ for some $a,b\in\mathbb{Z}_n$. The identity $xy\cdot y=yx\cdot x$ implies $[2a^2]_n=1$.
     Conversely, a groupoid $(Q,\cdot)$ satisfying these conditions is a $k$-tanslatable $ARO$-groupoid. From $[2a^2]_n=1$ it follows that $(a,n)=1$. Now, if $d|b$ and $d|n$, then, by $[a+bk]_n=0$, $d|a$, so $(b,n)=1$. Thus $(Q,\cdot)$ is a quasigroup.
\end{proof}

Now we describe several types of idempotent $k$-translatable quasigroups associated with
combinatorial designs \cite{CD}. First, we will describe idempotent $k$-translatable quasigroups 
satisfying the identity $(xy\cdot y)y=x$. Idempotent quasigroups satisfying this identity are called {\em $C3$ quasigroups} and correspond to a class of resolvable Mendelsohn triple systems (see, for example, \cite{BMM}). It was shown in \cite{Ben3} that finite $C3$ quasigroups exist only for orders $n=3t+1$. By Lemma \ref{L-41}, $k$-translatable $C3$ quasigroups have odd order $n$ and $(k,n)=1$.

\begin{proposition}\label{P-417} A naturally ordered groupoid $(Q,\cdot)$ of order $n$ is a $k$-translatable $C3$ quasigroup if and only if \ $x\cdot y=[ax+by]_n$ for some $a,b\in\mathbb{Z}_n$ such that $[a+b]_n=1$, $[a+bk]_n=0$ and $[a^3]_n=1$.
\end{proposition}
\begin{proof} 

Proof. By Theorem \ref{T-42}, the multiplication in a $k$-translatable $C3$ quasigroup is defined by $x\cdot y=[ax+by]_n$, where $a,b\in\mathbb{Z}_n$, $[a+b]_n=1$ and $[a+bk]_n=0$. From $(xy\cdot y)y=x$ it follows that $[a^3]_n=1$.
 
Conversely, a groupoid satisfying these conditions is an idempotent, $k$-trans\-la\-ta\-ble $C3$ groupoid. From $[a^3]_n=1$ it follows that it is right cancellative and, hence, it is a quasigroup. 
\end{proof}

\begin{corollary}\label{C-418} A commutative $C3$ quasigroup is $k$-translatable only for $k=6$ and has the form $x\cdot y=[4x+4y]_7$.
\end{corollary}
\begin{proof} In this case $a=b$, $[2a]_n=1$ and $[a^2+a+1]_n=0$. Thus, $0=[2(a^2+a+1)]_n=[a+3]_n$. So, $n=7$ and $k=6$. 
\end{proof}

\begin{proposition}\label{P-416} A naturally ordered groupoid $(Q,\cdot)$ of order $n$ is a $k$-translatable $ARO$-quasigroup if and only if $x\cdot y=[ax+by]_n$ for some  $a,b\in\mathbb{Z}_n$ such that $[a+b]_n=1$, $[a+bk]_n=0$ and $[2a^2]_n=1$.   
\end{proposition}
\begin{proof} If $(Q,\cdot)$ is a $k$-translatable $ARO$-quasigroup, then, by Theorem \ref{T-42}, $x\cdot y=[ax+by]_n$, $[a+b]_n=1$ and $[a+bk]_n=0$ for some $a,b\in\mathbb{Z}_n$. The identity $xy\cdot y=yx\cdot x$ implies $[2a^2]_n=1$.
     Conversely, a groupoid $(Q,\cdot)$ satisfying these conditions is a $k$-tanslatable $ARO$-groupoid. From $[2a^2]_n=1$ it follows that $(a,n)=1$. Now, if $d|b$ and $d|n$, then, by $[a+bk]_n=0$, $d|a$, so $(b,n)=1$. Thus $(Q,\cdot)$ is a quasigroup.
\end{proof}

A groupoid $(Q,\cdot)$ satisfying the identity $x\cdot xy=yx$ is called a {\it Stein groupoid}. Left cancellative Stein groupoids are idempotent and right cancellative. So, idempotent $k$-translatable Stein groupoids are idempotent $k$-translatable quasigroups. Stein quasigroups have applications in the theory of Latin squares (cf. \cite{Kee} and \cite{Scerb}) and combinatorial designs (cf. \cite{Ben}). 

Below we show that $k$-translatable Stein quasigroups are associated with {\it left modular groupoids}, i.e., groupoids satisfying the identity $x\cdot yz=z\cdot yx$. 
Their dual groupoids are associated with groupoids satisfying the identity $xy\cdot z=zy\cdot x$ and are called {\it right modular}.

By Lemma \ref{L-41}, $k$-translatable Stein quasigroups are of odd order $n$ and $(k,n)=1$. 

As a simple consequence of Theorem \ref{T-42} we obtain the following two lemmas.
\begin{lemma}\label{l-mod}
A naturally ordered idempotent $k$-translatable groupoid $(Q,\cdot)$ of order $n$ is left modular if and only if $x\cdot y=[ax+by]_n$ for some $a,b\in \mathbb{Z}_n$, $[a+b]_n=1$, $[a+bk]_n=0$ and $[a^2-3a+1]_n=0$.
\end{lemma}

\begin{lemma}\label{r-mod}
A naturally ordered idempotent $k$-translatable groupoid $(Q,\cdot)$ of order $n$ is right modular if and only if $x\cdot y=[ax+by]_n$ for some $a,b\in \mathbb{Z}_n$, $[a+b]_n=1$, $[a+bk]_n=0$ and $[a^2+a-1]_n=0$.
\end{lemma}
\begin{proposition}\label{P-421} 
A naturally ordered groupoid $(Q,\cdot)$ of order $n$ is a $k$-translatable Stein quasigroup if and only if $\,[k^2-k-1]_n=0$ and $x\cdot y=[(k+1)x-ky]_n$. 
\end{proposition}
\begin{proof}
Let $(Q,\cdot)$ be a $k$-translatable Stein quasigroup of order $n$. Then, by Theorem \ref{T-42}, $x\cdot y=[ax+by]_n$ for some $a,b\in \mathbb{Z}_n$, where $[a+b]_n=1$, $[a+kb]_n=0$, $[b^2]_n=a$ and $[a^2-3a+1]_n=0$.  From $0=[a+bk]_n$ we obtain $0=[ab+b^2k]_n=[a(b+k)]_n=[(a-3)a(b+k)]_n=[-(b+k)]_n$. Thus, $b=[-k]_n$. Therefore $x\cdot y=[(k+1)x-ky]_n$. Obviously $[k^2-k-1]_n=0$.

Conversely, a groupoid satisfying these conditions is an idempotent $k$-trans\-lata\-ble Stein groupoid. Since $[(k+1)(k-2)]_n=-1$, $(Q,\cdot)$ is right cancellative and, so, it is a quasigroup.
\end{proof}

\begin{corollary}
$k$-translatable Stein quasigroups cannot be commutative.
\end{corollary}

\begin{corollary}
A groupoid $(Q,\cdot)$ is a $k$-translatable Stein quasigroup if and only if it is an idempotent $k$-translatable left modular quasigroup.
\end{corollary}
\begin{corollary}
A groupoid $(Q,\cdot)$ of order $n$ is a $k$-translatable Stein quasigroup if and only if its dual groupoid is an idempotent $[k-1]_n$-translatable right modular quasigroup.
\end{corollary} 

As a consequence of the above results we obtain the following two theorems.

\begin{theorem}\label{T-425}
A naturally ordered groupoid $(Q,\cdot)$ of order $n$ with the multiplication defined by $x\cdot y=[ax+by]_n$, where $a,b\in \mathbb{Z}_n$, $[a+b]_n=1$ and $[a+bk]_n=0$
is a $k$-translatable quasigroup that is 
\begin{enumerate}
\item[$(1)$]  quadratical if and only if $[2a^2-2a+1]_n=0$,
\item[$(2)$]  hexagonal if and only if $[a^2-a+1]_n=0$,
\item[$(3)$]  $GS$-quasigroup if and only if $[a^2-a-1]_n=0$,
\item[$(4)$]  right modular quasigroup if and only if $[a^2+a-1]_n=0$,
\item[$(5)$]  left modular quasigroup if and only if $[a^2-3a+1]_n=0$,
\item[$(6)$]  Stein quasigroup if and only if $[a^2-3a+1]_n=0$,
\item[$(7)$]  $ARO$-quasigroup if and only if $[2a^2]_n=1$,
\item[$(8)$]  $C3$ quasigroup if and only if $[a^3]_n=1$.
\end{enumerate}
\end{theorem}

\begin{theorem}\label{T-426} A naturally ordered groupoid $(Q,\cdot)$ of order $n$ with the multiplication defined by $x\cdot y=[ax+by]_n$, where $a,b\in\mathbb{Z}_n$, $[a+b]_n=1$ and $[a+bk]_n=0$, is a $k$-transalatable
\begin{enumerate}
\item[$(1)$] quadratical quasigroup if and only if $k=[1-2a]_n$,
\item[$(2)$] hexagonal quasigroup if and only if $k=[1-a]_n$,
\item[$(3)$] $GS$-quasigroup if and only if $k=[a+1]_n$,
\item[$(4)$] right modular quasigroup if and only if $k=[-1-a]_n$,
\item[$(5)$] left modular quasigroup if and only if $k=[a-1]_n$,
\item[$(6)$] $ARO$-quasigroup if and only if $k=[-1-2a]_n$,
\item[$(7)$]  $C3$ quasigroup if and only if $[(1-a^2)k]_n=1$.
\end{enumerate}
\end{theorem}

\medskip

Quasigroups satisfying the identity \ $x(xy\cdot z)=(y\cdot zx)x$ are called {\it Cheban quasigroups}.

\begin{proposition} There are no idempotent $k$-translatable Cheban quasigroups. 
	\end{proposition}
\begin{proof}
The Cheban identity implies
$(1)$ $[2a^3-3a^2-a+1]_n=0$, $(2)$ $[a^3-3a^2+a]_n=0$ and $(3)$ $[a^3-2a+1]_n=0$. From $(2)$ and $(3)$, multiplying each identity by $a$, we obtain $(4)$ $[3a^3]_n=[3a^2-a]_n$. By $(2)$ and $(4)$, $[3a^3]_n=[a^3]_n$ and so $[2a^3]_n=0$. Multiplying $(2)$ by $2a$ we obtain $[2a^2]_n=0$. Multiplying $(2)$ by $2$ we get $[2a]_n=0$. Now, multiplying $(1)$ by $2$ we obtain $[2]_n=0$. So, $n=2$, a contradiction because $n$ must be odd.
\end{proof}

A {\it Schr\"oder quasigroup} is a quasigroup satisfying the identity \ $xy\cdot yx=x$.
\begin{proposition} There are no idempotent $k$-translatable Schr\"oder quasigroups.
\end{proposition}
\begin{proof} By Theorem \ref{T-42}, in an idempotent $k$-translatable Schr\"oder quasigroup of order $n$ must be $[2a^2]_n=[2a]_n$. Hence, multiplying $[a+bk]_n=0$ by $2a$ we obtain $[2a]_n=0$. Thus $2\cdot n=[2a+bn]_n=n=n\cdot n$, a contradiction. 
\end{proof}

From the above results it follows that the only idempotent $k$-translatable quasigroup that is $C3$ and quadratical is $x\cdot y=[3x+11y]_{13}$. The only idempotent $k$-translatable quasigroup that is $C3$ and $ARO$ is $x\cdot y=[2x+6y]_{7}$. The only idempotent $k$-translatable quasigroup that is
right modular and quadratical is $x\cdot y=[2x+4y]_{5}$ . The only idempotent $k$-translatable
quasigroup that is hexagonal and $ARO$ is $x\cdot y=[5x+3y]_{7}$. The only idempotent $k$-translatable
quasigroup that is the dual of a $C3$ quasigroup and $ARO$ is $x\cdot y=[27x+5y]_{31}$. There are no
idempotent $k$-translatable quasigroups that are $ARO$ and right modular, Stein and $C3$, Stein and right modular, right modular and $GS$, $ARO$ and Stein, $GS$ and $ARO$, Stein and $GS$, $GS$ and $C3$, Hexagonal and Stein, Hexagonal and $C3$, Hexagonal and right modular, hexagonal and $GS$, quadratical and hexagonal, quadratical and $ARO$, quadratical and $GS$, $C3$ and right modular, or quadratical and Stein.

\section{Parastrophes of $k$-translatable quasigroups}   %%%%%%%%%%%%%%%%%%%%%%%%%%%%%%%%%%%%%%%%   5

Each quasigroup $(Q,\cdot)$ determines five new quasigroups $Q_i=(Q,\circ_i)$ with the operations $\circ_i$ defined as follows:
$$
\begin{array}{cccc}
x\circ_1 y=z\;\Leftrightarrow\; x\cdot z=y,\\
x\circ_2 y=z\;\Leftrightarrow\; z\cdot y=x,\\
x\circ_3 y=z\;\Leftrightarrow\; z\cdot x=y,\\
x\circ_4 y=z\;\Leftrightarrow\; y\cdot z=x,\\
x\circ_5 y=z\;\Leftrightarrow\; y\cdot x=z.\\
\end{array}
$$
Such defined quasigroups are called {\em conjugates} or {\em parastrophes} of $(Q,\cdot)$. 
Since $x\circ_1 y=y\circ_4 x$, $x\circ_2 y=y\circ_3 x$ and $x\circ_5 y=y\cdot x$, parastrophes of a given quasigroups are pairwise dual. Moreover, parastrophes of some quasigroups are pairwise equal or isotopic (cf. \cite{Kee}, \cite{Lin} and \cite{Par}). Parastrophes of idempotent quasigroups are idempotent, but parastrophes of translatable quasigroups are not translatable in general. Indeed, a $3$-translatable quasigroup $Q$ defined by the multiplication table with the first row $1, 4, 3, 2, 8, 7, 6, 5$ has only one translatable parastrophe, namely its dual quasigroup, i.e., $(Q,\circ_5)$. In Theorem \ref{T-trans} we will prove that parastrophes of an idempotent $k$-translatable quasigroup also are idempotent and $k$-translatable, but for other value of $k$.

Let $Q_a=(Q,\cdot)$ be a naturally orderd idempotent and $k$-translatable quasigroup of order $n$ with the multiplication $x\cdot y=[ax+by]_n$, $a,b\in\mathbb{Z}_n$. Then, $[a+b]_n=1$ and $(a,n)=(b,n)=1$. The last means that $aa'+ns=1=bb'+nt$ for some $a',b',s,t$, i.e., $[aa']_n=[bb']_n=1$. It is clear that elements $a',b'\in\mathbb{Z}_n$ are uniquely determined. Moreover, $a'=b'$ is equivalent to $a=b$.

\begin{theorem}\label{T-51}
Multiplications of parastrophes of a quasigroup $Q_a$ have the form:
\begin{enumerate}
\item[$(1)$] $x\circ_1 y=[(1-b')x+b'y]_n$,
\item[$(2)$] $x\circ_2 y=[a'x+(1-a')y]_n$,
\item[$(3)$] $x\circ_3 y=[(1-a')x+a'y]_n$,
\item[$(4)$] $x\circ_4 y=[b'x+(1-b')y]_n$,
\item[$(5)$] $x\circ_5 y=[bx+ay]_n$,
\end{enumerate}
where $[aa']_n=[bb']_n=1$.
\end{theorem}
\begin{proof}
$(1)$. By definition, $x\circ_1 y=z\Leftrightarrow x\cdot z=y\Leftrightarrow [ax+bz]_n=y$, and consequently, $[bz]_n=[y-ax]_n$. From this, multiplying by $b'$, we obtain 
$z=[b'y-ab'x]_n=[-ab'x+b'y]_n$. So, 
$$
x\circ_1 y=[-ab'x+b'y]_n=[(1-b')x+b'y]_n .
$$

$(2)$. Similarly, 
$x\circ_2 y=z\Leftrightarrow z\cdot y=x\Leftrightarrow [az+by]_n=x$. Whence, 
$$
x\circ_2 y=[a'x-a'by]_n=[a'x+a'(a-1)y]_n=[a'x+(1-a')y]_n.
$$

Analogously we can prove $(3)$ and $(4)$. $(5)$ is obvious.
\end{proof}

\begin{corollary}
All parastrophes of a quasigroup $Q_a$ are isotopic to the group $\mathbb{Z}_n$.
\end{corollary}

\begin{theorem}\label{T-trans}
A quasigroup $Q_a$ is $k$-translatable for $k=[1-b']_n$. Its parastrophes are $k^*$-translatable, where
\begin{enumerate}
\item[$(1)$] $k^*=a$ for $(Q,\circ_1)$,
\item[$(2)$] $k^*=[1-k]_n$ for $(Q,\circ_2)$,    %$[bk^*]_n=1$ for $(Q,\circ_2)$,
\item[$(3)$] $k^*=[1-a]_n$ for $(Q,\circ_3)$,
\item[$(4)$] $k^*=a'$ for $(Q,\circ_4)$,     %$[ak^*]_n=1$ for $(Q,\circ_4)$,
\item[$(5)$] $k^*=[1-a']$ for $(Q,\circ_5)$.     %$[kk^*]_n=1$ for $(Q,\circ_5)$.
\end{enumerate}
\end{theorem}
\begin{proof}
From $[a+(1-b')b]_n=[a+b-1]_n=0$ it follows that a quasigroup $Q_a$ is $k$-translatable for $k=[1-b']_n$. Similartl, from $x\circ_1 y=[(1-b')x+b'y]_n$ and  $[(1-b')+b'a]_n=[(1-a')+b'(1-b)]_n=0$ it follows that $(Q,\circ_1)$ is $k^*$-translatable for $k^*=a$.

Analogously we can verify other cases.
\end{proof}

\begin{lemma}\label{L-par} 
For a quasigroup $Q_a$ and its parastrophes the following relationships are possible:
\begin{enumerate}
\item[$(1)$] $Q_a=Q_1\Leftrightarrow x\cdot y=[2x-y]_n$,
\item[$(2)$] $Q_a=Q_2\Leftrightarrow x\cdot y=[-x+2y]_n$, 
\item[$(3)$] $Q_a=Q_3\Leftrightarrow [ab]_n=1$,
\item[$(4)$] $Q_a=Q_4\Leftrightarrow [ab]_n=1$,
\item[$(5)$] $Q_a=Q_5\Leftrightarrow a=b$.
\end{enumerate}
\end{lemma}
\begin{proof}
Indeed, $Q_a=Q_1$ implies $b=b'$. Thus $1=[bb']_n=[1-2a+a^2]_n$, i.e., $[a^2]_n=[2a]_n$, whence, multiplying by $a'$, we obtain $a=2$, and consequently $b=n-1$. So, $x\cdot y=[2x-y]_n$. The converse statement is obvious.

Other statements can be proved analogously.
\end{proof}

\begin{theorem}
For parastrophes of a quasigroup $Q_a$ the following cases are possible:
\begin{enumerate}
\item[$(a)$] $Q_a=Q_1=Q_2=Q_3=Q_4=Q_5\Leftrightarrow %a=2,\ n=3 \Leftrightarrow 
x\cdot y=[2x+2y]_3$,
\item[$(b)$] $ Q_a=Q_3=Q_4\ne Q_1=Q_2=Q_5\Leftrightarrow [ab]_n=1, \ a\ne b$,
\item[$(c)$] $Q_a=Q_1\ne Q_2=Q_3\ne Q_4=Q_5\Leftrightarrow %a=2, \ n>3 \Leftrightarrow 
x\cdot y=[2x-y]_n$ and $n>3$,
\item[$(d)$] $Q_a=Q_2\ne Q_1=Q_4\ne Q_3=Q_5\Leftrightarrow %b=2, \ n>3 \Leftrightarrow 
x\cdot y=[-x+2y]_n$ and $n>3$,
\item[$(e)$] $Q_a=Q_5\ne Q_1=Q_3\ne Q_2=Q_4\Leftrightarrow a=b$ and $n>3$,
\item[$(f)$] $ Q_a\ne Q_1\ne Q_2\ne Q_3\ne Q_4\ne Q_5\Leftrightarrow a\ne b\ne 2$ and $n>3$.
\end{enumerate}
\end{theorem}

\begin{proof}
Let $Q_a$ be an idempotent $k$-translatable quasigroup of order $n$. Then $n$ is odd. For $n=3$ we have only one possibility: $x\cdot y=[2x+2y]_3$. Then, as it is not difficult to see, all parastrophes of $Q_a$ are equal to $Q_a$. Conversely, by Lemma \ref{L-par}, $Q_a=Q_1=Q_2$ implies $[2x-y]_n=[-x+2y]_n$. Thus $n=3$ and $x\cdot y=[2x+2y]_3$. This proves $(a)$.

Now let $n>3$. Then, by Lemma \ref{L-par}, $Q_a=Q_3\Leftrightarrow [ab]_n=1\Leftrightarrow a'=b,\, b'=a\Leftrightarrow Q_a=Q_3=Q_4\ne Q_1=Q_2=Q_5$ and $a\ne b$, since $a=b$ implies $n=3$. This proves $(b)$. If $Q_a=Q_1$, then, by Lemma \ref{L-par}, $x\cdot y=[2x-y]_n$. Then obviously, $Q_a=Q_1\ne Q_2=Q_3\ne Q_4=Q_5$. This proves $(c)$.
The proof of $(d)$ is analogous. The case $(e)$ is obvious. So, $Q_a=Q_i$ implies one of statements $(a)-(d)$.

To prove $(f)$ suppose that $Q_a\ne Q_i$ for all $i=1,\ldots,5$. Then also $Q_s\ne Q_t$ for all $1\leqslant s<t\leqslant 5$. Indeed, $Q_1=Q_2$ means that $1=[a'+b']_n$, whence, multiplying by $ab$, we obtain $[ab]_n=[b+a]_n=1$. So, by Lemma \ref{L-par}, $Q_a=Q_3$. $Q_1=Q_3$ implies $Q_a=Q_5$. $Q_1=Q_4$ gives $1=[2b']_n$, and consequently, $b=2$, $a=n-1$. Thus $Q_a=Q_2$, by Lemma \ref{L-par}. From $Q_1=Q_5$ we obtain $a=b$ and $a'=b'$. This implies $Q_1=Q_3$, and in the consequence, $Q_a=Q_5$. In the case $Q_2=Q_3$ we have $1=[2a']_n$, whence $a=2$, $b=n-1$ and $Q_a=Q_1$ by Lemma \ref{L-par}. $Q_2=Q_4$ implies $Q_a=Q_5$. If $Q_2=Q_5$, then $a'=b$, $b'=a$. Consequently $Q_a=Q_4$. Further, $Q_3=Q_4$ implies $[a'+b']_n=1$, whence, as in the case $Q_1=Q_2$, we obtain $Q_a=Q_3$. $Q_3=Q_5$ and $Q_4=Q_5$ imply $Q_a=Q_2$ and $Q_a=Q_1$, respectively.
Thus in any case $Q_s=Q_t$ implies $Q_a=Q_i$ for some $i=1,2,\ldots,5$. Hence, if $Q_a\ne Q_i$ for all $i=1,\ldots,5$, then also $Q_s\ne Q_t$ for all $1\leqslant s<t\leqslant 5$. This proves $(f)$.
The example $x\cdot y=[3x+9y]_{11}$ shows that the case $(f)$ is possible.
\end{proof}

Using Theorems \ref{T-425}, \ref{T-426} and \ref{T-trans} we can calculate the $k^*$-translatability of the parastrophes of several types of idempotent $k$-translatable quasigroups as a function of $a$. Results of calculations are presented below.

$$\begin{array}{|c|c|c|c|c|c|c|c|c|}\hline
&\!quadratical\!&\!hexagonal\!&GS&ARO&Stein&\!right\;modular\!\\ \hline
\rule{0pt}{10pt}Q_a&[1-2a]_n&[1-a]_n&[a+1]_n&[-1-2a]_n&[a-1]_n&[-1-a]_n\\ \hline
Q_1&a&a&a&a&a&a\\ \hline
\rule{0pt}{10pt}Q_2&[2a]_n&a&[-a]_n&[2+2a]_n&[2-a]_n&[a+2]_n\\ \hline
\rule{0pt}{10pt}Q_3&[1-a]_n&[1-a]_n&[1-a]_n&[1-a]_n&[1-a]_n&[1-a]_n\\ \hline
\rule{0pt}{10pt}Q_4&[2-2a]_n&[1-a]_n&[a-1]_n&[2a]_n&[3-a]_n&[a+1]_n\\ \hline
\rule{0pt}{10pt}Q_5&[2a-1]_n&a&[2-a]_n&[1-2a]_n&[a-2]_n&[-a]_n\\ \hline
\end{array}
$$
\centerline{\it Table $1$. Translatability as a function of $a$.}

\medskip
If $Q_a$ is $k$-translatable, then its parastrophes are $k^*$-translatable for the following values of $k^*$.
$$
\begin{array}{|c|c|c|c|c|c|c|c|c|}\hline
Q_a&\!quadratical\!&\!hexagonal\!&GS&ARO&Stein&\!right\;modular\!\\ \hline
\rule{0pt}{10pt}Q_1&\![1\!-\!k\!-\!a]_n\!&[1-k]_n&[k-1]_n&\![\!-1\!-\!k\!-\!a]_n\!&[k+1]_n&[-1-k]_n\\ \hline
\rule{0pt}{10pt}Q_2&[1-k]_n&[1-k]_n&[1-k]_n&[1-k]_n&[1-k]_n&[1-k]_n\\ \hline
\rule{0pt}{10pt}Q_3&[k+a]_n&k&\![k\!-\!2a]_n&[k\!+\!a\!+\!2]_n&[-k]_n&[k+2]_n\\ \hline
\rule{0pt}{10pt}Q_4&[k+1]_n&k&[k-2]_n&[-1-k]_n&[2-k]_n&[-k]_n\\ \hline
\rule{0pt}{10pt}Q_5&[-k]_n&[1-k]_n&[3-k]_n&[k+2]_n&[k-1]_n&[k+1]_n\\ \hline
\end{array}
$$
\centerline{\it Table $2$. Tranlatability of parastrophes as a function of $k$.}

Using Theorem \ref{T-425} and Theorem \ref{T-51}, we can readily see that all parastrophes of a hexagonal quasigroup $Q_a$ are hexagonal. Moreover, a quasigroup $Q_a$ is hexagonal if and only if one of its parastrophes is hexagonal. In other types of idempotent $k$-translatable quasigroups the situation is more complicated. For example, parastrophes of a quadratical $k$-translatable quasigroup $Q_a$ are quadratical only for some values of $a$ and $n$.
$$
\begin{array}{|c|c|c|c|c|c|c|c|}\hline
Q_a&\!\!quadratical\!\!&GS&ARO&Stein&\!r.\,modular\!&C3\\ \hline
\rule{0pt}{10pt}Q_1&\!a\!=\!2,\,n\!=\!5\!& never &\!a\!=\!2,\,n\!=\!7\!& never & always &\!a\!=\!2,\,n\!=\!7\!\\ \hline
\rule{0pt}{10pt}Q_2&\!a\!=\!4,\,n\!=\!5\!&never&never&always&never&always\\ \hline
\rule{0pt}{10pt}Q_3&\!a\!=\!4,\,n\!=\!5\!&never&\!a\!=\!5,\,n\!=\!7\!&never&never&\!a\!=\!2,\,n\!=\!7\!\\ \hline
\rule{0pt}{10pt}Q_4&\!a\!=\!2,\,n\!=\!5\!&never&\!a\!=\!5,\,n\!=\!7\!&never&never&\!a\!=\!4,\,n\!=\!7\!\\ \hline
\rule{0pt}{10pt}Q_5&always&always&never&never&never&\!a\!=\!4,\,n\!=\!7\!\\ \hline
\end{array}
$$
\centerline{\it Table $3$. Parastrophe types}

\noindent
W.A. Dudek \\
 Faculty of Pure and Applied Mathematics,\\
 Wroclaw University of Science and Technology,\\
 50-370 Wroclaw,  Poland \\
 Email: wieslaw.dudek@pwr.edu.pl\\[4pt]
R.A.R. Monzo\\
Flat 10, Albert Mansions, Crouch Hill,\\ London N8 9RE, United Kingdom\\
E-mail: bobmonzo@talktalk.net


\small
\begin{thebibliography}{20}
\bibitem{Ben3} {\bf F.E. Bennett},{\it On a class of $n^2\times 4$ orthogonal arrays and associated quasigroups},
Congressus Numerantium {\bf 39} (1983), $117-122$.
\bibitem{Ben} {\bf F.E. Bennett}, {\it The spectra of a variety of quasigroups and related combinatorial designs}, Discrete Math. {\bf 77} (1989), $29-50$. 
\bibitem{BMM} {\bf F.E. Bennet, E. Mendelsohn and N.S. Mendelsohn}, {\it Resolvable perfect cyclic designs}, J. Combin. Theory (A), {\bf 29} (1980), $142-150$.
\bibitem{1} {\bf A.H. Clifford and G.B. Preston},  {\it The algebraic theory of semigroups}, vol. 1, Math. Surveys 7, AMS, 1961.   
\bibitem{CD} {\bf Ch.J. Colbourn and J. Dinitz}, (editors), {\it The CRC handbook of combinatorial designs}, 2nd ed., Discrete Math. Appl., Chapman and Hall/CRC (2007).
\bibitem{Par} {\bf W.A. Dudek}, {\it Quadratical quasigroups}, Quasigroups and Related Systems {\bf 23} (2015), $221-230.$
\bibitem{2} {\bf W.A. Dudek}, {\it Parastrophes of quasigroups}, Quasigroups and Related Systems {\bf 4} (1997), $9-13.$
\bibitem{3} {\bf W.A. Dudek and R.A.R. Monzo}, {\it On the fine structure of quadratical quasigroups}, Quasigroups and Related Systems {\bf 24} (2016), $205-218$.  
\bibitem{4} {\bf W.A. Dudek and R.A.R. Monzo}, {\it Translatable quadratical quasigroups}, (submitted),
arXiv:1708.08830v1 [math.RA] 27 Aug 2017.
\bibitem{5} {\bf W.A. Dudek and R.A.R. Monzo}, {\it Translatability and translatable semigroups}, (submitted).
\bibitem{Kee} {\bf D. Keedwell and J. D\'enes}, {\it Latin squares and their applications}, Noth-Holland, 2015. 
\bibitem{Kol} {\bf Z. Kolar-Begovi\'c and V. Volenec}, {\it Affine regular dodecahedron in $GS$-quasi\-groups}, Quasigroups and Related Systems {\bf 13} (2005), $229-236$.
\bibitem{6} {\bf M.V. Lawson}, {\it Generalized heaps, inverse semigroups and Morita equivalence}, Algebra Universalis {\bf 66} (2011), $317-330$. 
 \bibitem{Lin} {\bf C.C. Lindner and D. Steedly}, {\em On the number of conjugates of a quasigroups}, Algebra Universalis
{\bf 5} (1975), $191-196$.
\bibitem{7} {\bf R.A.R. Monzo}, {\it The ternary operations of groupoids}, arXiv:1510.07955v1 [math.GM]
\bibitem{Scerb} {\bf V. Shcherbacov}, {\it Elements of quasigroup theory and applications}, Champan and Hall Book, 2017.
%\bibitem{8} {\bf S. Vidak}, {\it Pentagonal quasigroups}, Quasigroups and Related Systems {\bf 22} (2014), $147-158.$
\bibitem{Vol} {\bf V. Volenec}, {\it $GS$ quasigroups}, \v{C}as. pest. mat. {\bf 115} (1990), $307-318$.
\bibitem{8a} {\bf V. Volenec}, {\it Hexagonal quasigroups}, Arch. Math. (Brno), {\bf 27} (1991), $113-122$.
\bibitem{9} {\bf V. Volenec}, {\it Quadratical groupoids}, Note di Matematica {\bf 13} (1993), $107-115$. 
\bibitem{10} {\bf V. Volenec}, {\it Heaps and right solvable Ward groupoids}, J. Algebra {\bf 156} (1993), $1-4$.
\bibitem{11} {\bf V. Volenec, Z. Kolar-Begovi\'c and R. Kolar-\v{S}uper}, {\it $ARO$-quasigroups}, Quasigroups and Related Systems {\bf 18} (2010), $213-228.$ 
\end{thebibliography}
\end{document}